\documentclass{amsart}
\usepackage{amscd,amsmath,mathrsfs,amssymb}
\usepackage{mathtools}
\usepackage[utf8]{inputenc}
\usepackage[T1]{fontenc}
\usepackage{tikz-cd}
\usepackage{hyperref}

\definecolor{mygreen}{rgb}{0,0.6,0}

\newtheorem{thm}{Theorem}[section]

\newtheorem{lem}[thm]{Lemma}
\newtheorem{cor}[thm]{Corollary}

\theoremstyle{definition}

\newtheorem{fact}[thm]{Fact}

\newtheorem{remark}[thm]{Remark}

\newtheorem{exampleA}[thm]{Example}
\newtheorem{questA}[thm]{Question}
\newtheorem{A}[thm]{}

\newtheorem*{ac}{Acknowledgments}

\makeatletter
\newcommand{\hocolim@}[2]{%
  \vtop{\m@th\ialign{##\cr
    \hfil$#1\operator@font hocolim$\hfil\cr
    \noalign{\nointerlineskip\kern1.5\ex@}#2\cr
    \noalign{\nointerlineskip\kern-\ex@}\cr}}%
}

\newcommand{\dhocolim}{%
  \mathop{\mathpalette\hocolim@{\rightarrowfill@\textstyle}}\nmlimits@
}
\newcommand{\EMP}{\textbf}
\newcommand*{\Perp}[1]{{}^{\perp_{#1}}}
\DeclareMathOperator*{\Spec}{Spec}
\DeclareMathOperator*{\Supp}{Supp}

\DeclareMathOperator*{\Cat}{Cat}

\DeclareMathOperator*{\CAT}{CAT}

\DeclareMathOperator*{\holim}{holim}
\DeclareMathOperator*{\hocolim}{hocolim}
\DeclareMathOperator*{\Mlim}{Mlim}
\DeclareMathOperator*{\Mcolim}{Mcolim}

\DeclareMathOperator*{\ModR}{Mod-\textit{R}}

\DeclareMathOperator*{\op}{op}

\def\RGamma{\operatorname{\mathbf{R}\Gamma}}

\def\Hom{\operatorname{Hom}}
\def\RHom{\operatorname{\mathbf{R}Hom}} 
\def\Rlim{\operatorname{\mathbf{R}lim}}
\def\Lcolim{\operatorname{\mathbf{L}colim}}
\def\Der{\operatorname{\mathbf{D}}}
\def\Ch{\operatorname{\mathbf{C}}}
\def\dia{\operatorname{dia}}
\newcommand{\DD}{\ensuremath{\mathbb{D}}}
\newcommand{\Z}{\ensuremath{\mathbb{Z}}}
\newcommand{\pp}{\mathfrak{p}}
\newcommand{\qq}{\mathfrak{q}}

\newcommand{\Gcal}{\ensuremath{\mathcal{G}}}

\newcommand{\Scal}{\ensuremath{\mathcal{S}}}
\newcommand{\Ucal}{\ensuremath{\mathcal{U}}}
\newcommand{\Vcal}{\ensuremath{\mathcal{V}}}
\newcommand{\Wcal}{\ensuremath{\mathcal{W}}}
\newcommand{\Tcal}{\ensuremath{\mathcal{T}}}

\newcommand{\Lcal}{\ensuremath{\mathcal{L}}}
\newcommand{\Ecal}{\ensuremath{\mathcal{E}}}
\newcommand{\Ccal}{\ensuremath{\mathcal{C}}}

\newcommand{\Hcal}{\ensuremath{\mathcal{H}}}
\newcommand{\Xscr}{\ensuremath{\mathscr{X}}}

\newcommand{\Vscr}{\ensuremath{\mathscr{V}}}

\title[
Telescope conjecture for homotopically smashing t-structures]{Telescope conjecture for homotopically smashing t-structures over commutative noetherian rings}

\author{Michal Hrbek}
\address[M. Hrbek]{Institute of Mathematics of the Czech Academy of Sciences, \v{Z}itn\'{a} 25, 115 67 Prague, Czech Republic}
\email{hrbek@math.cas.cz}

\author{Tsutomu Nakamura}
\address[T. Nakamura]{Graduate School of Mathematics, Nagoya University, Furocho, Chikusaku, Nagoya 464-8602, Japan} 
\email{tsutomu.nakamura@math.nagoya-u.ac.jp}

\thanks{The first author was supported by the GAČR project 20-13778S and RVO: 67985840. The second author was supported by the Program Ricerca di Base 2015 of the University of Verona.}
\keywords{Derived category; telescope conjecture; t-structure; smashing; derivator; cosilting object.}
\subjclass[2010]{Primary 13D09. Secondary 18E30.}

\begin{document}
\maketitle

\begin{abstract}
	We show that any homotopically smashing t-structure in the derived category of a commutative noetherian ring is compactly generated. This generalizes the validity of the telescope conjecture for commutative noetherian rings due to Neeman. As another consequence, we obtain a cofinite type result for pure-injective cosilting objects. We also give a formulation of telescope conjecture for homotopically smashing t-structures in underlying triangulated categories of certain Grothendieck derivators.
\end{abstract}


\section{Introduction}
	A Bousfield localization of a triangulated category is called smashing if it commutes with all coproducts. The telescope conjecture asks whether any such smashing localization is generated by compact objects. The question was originally asked by Ravenel for the stable homotopy category of spectra \cite{R}, and in this context the question remains open. For algebraic triangulated categories however, several strong results have been obtained. Namely, in the setting of derived categories of modules, the telescope conjecture was established for an arbitrary commutative noetherian ring by Neeman \cite{N1}. Among other results in this direction, Krause and \v{S}\v{t}ov\'{i}\v{c}ek \cite{KS} showed that the telescope conjecture holds in the derived category of a (one-sided) hereditary ring, and a ring-theoretic criterion equivalent to the telescope conjecture was given for commutative rings of weak global dimension one by Bazzoni and \v{S}\v{t}ov\'{i}\v{c}ek \cite{BS}.
On the other hand, examples of rings for which the question has a negative answer have been found, the first one is due to Keller \cite{Ke}.

	The notion of a t-structure was introduced by Be\u{\i}linson, Bernstein, and Deligne \cite{BBD} as a general framework for constructing cohomological functors from triangulated categories to abelian categories. In derived categories, the t-structures provide a natural habitat for tilting theory, that is, for the study of derived equivalences and their realizations, see e.g. \cite{AS}, \cite{NSZ}, \cite{PV}, and \cite{AMV}. Similarly to the case of Bousfield localization, the theory becomes much more tractable if we restrict it to t-structures generated by compact objects. In particular, such t-structures sometimes allow for a full classification. For commutative noetherian rings, a bijective correspondence between compactly generated t-structures and filtrations of the Zariski spectrum by supports was established by Alonso Tarr\'{i}o, Jerem\'{i}as L\'{o}pez and Saor\'{i}n \cite{AJS}, see also Theorem~\ref{T:AJS}. This was further generalized to arbitrary commutative rings \cite{Hcom}.

	Since Bousfield localizations correspond to stable t-structures (= triangulated t-structures), it is natural to look for a more general formulation of the telescope conjecture suitable for t-structures which are not necessarily stable. The condition of preserving coproducts for localizations does not translate well into the case of non-stable t-structures. In fact, any (possibly non-hereditary) torsion pair in the module category induces a t-structure in the derived category such that the coaisle is closed under coproducts (\cite[Example 6.2]{SSV}), and therefore even over the ring of integers, there is a proper class of such t-structures (see \cite[Theorem 4.1]{GS}). 

	Recently, Saor\'{i}n, \v{S}\v{t}ov\'{i}\v{c}ek, and Virili introduced in \cite{SSV} the homotopically smashing t-structures, that is, t-structures whose coaisles are closed under directed homotopy colimits inside a Grothendieck derivator. Under mild assumptions on the derivator, homotopically smashing t-structures encompass two well-studied situations. For stable t-structures, this condition is known to be equivalent to the induced localization functor being smashing (see \ref{SS:stable}). For t-structures which are non-degenerate, this condition characterizes the case when the t-structure is induced by a pure-injective cosilting object, giving a strong relation to tilting theory; such a fact was recently shown by Laking \cite{L}. Furthermore, if a t-structure is  homotopically smashing and non-degenerate, then its heart admits exact direct limits \cite[Theorem B]{SSV}, and is often a Grothendieck category (\cite[Theorem C]{SSV}, \cite[Corollary 3.8]{AMV2}, \cite[Theorem 4.6]{L}). 

	In this way, we arrive at the following natural generalization of the telescope conjecture: When is it true that every homotopically smashing t-structure is generated by compact objects? As seen from the previous paragraph, this question naturally extends the telescope conjecture for stable t-structures, and it can yield a cofinite type result for pure-injective cosilting objects for non-degenerate t-structures (see Corollary~\ref{C:cosilting} and \ref{SS:cosilting}). We remark that by extending the scope of the question to non-stable t-structures, a larger supply of counterexamples is available, see \cite[Corollary 8.5]{BH}. Nevertheless, the aim of this paper is to show that the answer to this more general question is still affirmative in the case of a commutative noetherian ring. That is, we prove the following:
\begin{thm}\label{T01}
	Let $R$ be a commutative noetherian ring. Then any homotopically smashing t-structure in the unbounded derived category $\Der(\ModR)$ of $R$ is compactly generated.
\end{thm}

	A version of Theorem~\ref{T01} restricted to t-structures which are cohomologically bounded below was proved in \cite[Theorem 4.2]{Hcom} by a technique which does not seem to generalize for unbounded complexes. The direction in which we aim to proceed here would be close to the way of Neeman's proof of the stable case in \cite{N1}. However, there are several differences which we should emphasize. 
	
	As shown in \cite{N1}, the localizing subcategories of $\Der(\ModR)$ bijectively correspond to the subsets of the Zariski spectrum, and the telescope conjecture follows as a consequence. 
For arbitrary t-structures, no such reasonable classification can exist, as the derived category $\Der(\ModR)$ usually has a proper class of them. 
Alternatively, there is a classification of compactly generated t-structures in $\Der(\ModR)$ in terms of filtrations of specialization-closed subsets of the Zariski spectrum (\cite{AJS}), although this classification does not directly tell us whether any homotopically smashing t-structure in $\Der(\ModR)$ is compactly generated.
Actually, to prove Theorem \ref{T01}, we need to concentrate more on the coaisle --- the right hand subcategory of the t-structure in question --- and one of the main steps is to show that these are cogenerated by shifts of stalk complexes of indecomposable injective modules.
Moreover, our proof of the theorem can be completed thanks to a result of Foxby and Iyengar \cite{FI} dealing with infima of unbounded complexes. In the proof of Lemma \ref{L06}, their result successfully works as a suitable method to solve a difficulty of non-stable t-structures.

The main aim of the next section is to prove Theorem~\ref{T01}. 
Motivated by this result, we will give in the appendix  a formulation of telescope conjecture for homotopically smashing t-structures in underlying triangulated categories of certain Grothendieck derivators. For this reason, the reader is referred to the appendix for most of the details and references concerning homotopy (co)limits and homotopically smashing t-structures.

\medskip
\noindent {\bf Convention.} Throughout this paper, subcategories of a given category are assumed to be full, additive and strict.


\section{Homotopically smashing t-structures over commutative noetherian rings}\label{S:rings}


We start with preparation of basic notations and definitions. Let $\Tcal$ be a triangulated category with all set-indexed products and coproducts (for example, the unbounded derived category $\Der(\ModR)$ of a ring $R$, or the underlying category $\DD(\star)$ of a stable derivator $\DD$). 
For a class $\Ccal$ of objects in $\Tcal$ and a subset $B$ of $\Z$, we define the following subcategories:
$$\Ccal\Perp{B} = \{X \in \Tcal \mid \Hom{}_\Tcal(C,X[i]) = 0 ~\forall C \in \Ccal, ~\forall i \in B\}$$
and
$$\Perp{B}\Ccal = \{X \in \Tcal \mid \Hom{}_\Tcal(X,C[i]) = 0 ~\forall C \in \Ccal, ~\forall i \in B\}.$$
\noindent The role of the symbol $B$ will be played by symbols $0$, $\leq 0$, $>0$ interpreted as subsets of $\Z$ in the obvious way.
A pair $(\Ucal,\Vcal)$ of subcategories of $\Tcal$ is called a \EMP{t-structure} in $\Tcal$ if the following axioms are satisfied:
\begin{enumerate}
	\item[(i)] $\Hom_{\Tcal}(U,V) = 0$ for all $U \in \Ucal$ and $V \in \Vcal$,
	\item[(ii)] for any object $X \in \Tcal$ there is a triangle $U \rightarrow X \rightarrow V \rightarrow U[1]$
in $\Tcal$ with $U \in \Ucal$ and $V \in \Vcal$, and
	\item[(iii)] $\Ucal$ is closed under the suspension functor $[1]$.
\end{enumerate}

\noindent We call $\Ucal$ the \EMP{aisle} of the t-structure $(\Ucal, \Vcal)$. Moreover, given a subcategory $\Ucal$ of $\Tcal$, we call $\Ucal$ an aisle if there is a t-structure having $\Ucal$ as the aisle. Analogously, the same custom is applied for the subcategory $\Vcal$ and the term \EMP{coaisle}.
Note that any aisle is closed under direct summands and all coproducts in $\Tcal$. Dually, any coaisle is closed under direct summands and under all products in $\Tcal$. Furthermore, any coaisle  $\Vcal$ is closed under the cosuspension functor $[-1]$. For more details about t-structures in triangulated categories, see \cite{BBD} and \cite{KV}.

Denote by $\Tcal^c$ the subcategory of $\Tcal$ consisting of all \EMP{compact} objects, that is, objects $C \in \Tcal$ such that $\Hom_\Tcal(C,-)$ naturally preserves all coproducts in $\Tcal$. Suppose that $\Tcal^c$ is skeletally small.
We say that a triangulated category $\Tcal$ is \EMP{compactly generated} if $X\in (\Tcal^c)^{\perp_0}$ implies $X = 0$ for any $X \in \Tcal$.
A t-structure $(\Ucal,\Vcal)$ in $\Tcal$ is \EMP{compactly generated} if  $\Vcal = \Ccal\Perp{0}$ for some subcategory $\Ccal$ of $\Tcal^c$.  For a resource on compactly generated t-structures suitable for our application, we refer the reader to \cite{AJS} and the references therein.


In this section, we will work in the unbounded derived category $\Der(\ModR)$ of the module category $\ModR$ of a commutative noetherian ring $R$. 
As explained in the appendix, we consider $\Der(\ModR)$ as the underlying category of the standard derivator $\DD_{\ModR}$ of the Grothendieck category $\ModR$, which in particular allows us to talk about directed homotopy colimits and homotopically smashing t-structures in $\Der(\ModR)$.
However, in this setting we may also use the more direct equivalent definition explained below; Theorem~\ref{T01} will be proved based on it.

\begin{fact}\label{fact} A subcategory $\mathcal{C}$ of $\Der(\ModR)$ is closed under directed homotopy colimits if and only if for any directed diagram $\{X_i\}_{i\in I}$ in the category $\Ch(\ModR)$ of cochain complexes such
that all components $X_i$ are objects in $\mathcal{C}$, the direct limit
$\varinjlim_{i\in I} X_i$ in $\Ch(\ModR)$ belongs to $\mathcal{C}$.
Therefore, a homotopically smashing t-structure in $\Der(\ModR)$ is a t-structure $(\Ucal,\Vcal)$ whose coaisle $\Vcal$ is closed under direct limits of its objects in the level of $\Ch(\ModR)$. For more details, see Example \ref{example derivator}. 
\end{fact}

Let $\Spec(R)$ denote the Zariski spectrum of $R$. The support of an $R$-module $M$ is defined as $\Supp M=\{\pp\in \Spec (R) \mid M_\pp\neq 0\}$. 
A subset $V$ of $\Spec (R)$ is \EMP{specialization closed} if 
it is a union of Zariski closed subsets of $\Spec (R)$.

An \EMP{sp-filtration} is a map $\phi: \Z \rightarrow 2^{\Spec(R)}$, such that $\phi(i)$ is a specialization closed subset of $\Spec(R)$ for each $i \in \Z$, and such that $\phi$ is decreasing, that is, $\phi(i) \subseteq \phi(j)$ if $j \leq i$. 
Using this notion, Alonso Tarr\'{i}o, Jerem\'{i}as L\'{o}pez and Saor\'{i}n classified the compactly generated t-structures in $\Der(\ModR)$.

\begin{thm}\label{T:AJS}\emph{(\cite[Theorem 3.10, Theorem 3.11]{AJS})}
	Let $R$ be a commutative noetherian ring. Then there is a bijective correspondence

$$\left \{ \begin{tabular}{ccc} \text{ sp-filtrations $\phi$ } \\ \text{of $\Spec(R)$} \end{tabular}\right \} \leftrightarrow   \left \{ \begin{tabular}{ccc} \text{ compactly generated } \\ \text{ t-structures $(\Ucal,\Vcal)$ in $\Der(\ModR)$ } \end{tabular}\right \}\\ $$
	given by the assignments
	$$\phi \mapsto (\Ucal_\phi,\Vcal_\phi),$$
	where
	$\Ucal_{\phi} = \{X \in \Der(\ModR) \mid \Supp H^n(X) \subseteq \phi(n)~\forall n \in \Z\}$ and $\Vcal_{\phi} = \Ucal_{\phi}^{\Perp{0}}.$
	\end{thm}
	
\begin{remark}\label{R:compact}
For the reader's sake, we provide an explicit description of compact generators of the t-structure $(\Ucal_\phi,\Vcal_\phi)$.
For any ideal $I$ of $R$, let $K(I)$ denote the Koszul complex defined on some fixed finite generating set of $I$. Note that any Koszul complex is a compact object of $\Der(\ModR)$.
Set
$$\Scal_{\phi} = \{K(I)[-n] \mid \text{ $n \in \Z$, $V(I) \subseteq \phi(n)$}\},$$
where $V(I)=\{ \pp\in \Spec (R) \mid I \subseteq \pp \}$. 
It then holds that $\Vcal_{\phi} = \Scal_{\phi}\Perp{0}$
 by \cite[Corollary 3.9]{AJS} and the proof of \cite[Theorem 3.11]{AJS}. In fact, $\Ucal_\phi=\Perp{0}(\Scal_{\phi}\Perp{0})$ is the smallest aisle containing $\Scal_{\phi}$, see \cite[1.1 and 1.2]{AJS}.
We remark that this description also allows for classification of compactly generated t-structure over commutative rings which are not necessarily noetherian, see \cite[\S 5]{Hcom}.
\end{remark}

The next corollary is a key to our aim.
Note that a stalk complex means a complex concentrated in degree zero.

\begin{cor}\label{C01}
	Let $R$ be a commutative noetherian ring and $(\Ucal,\Vcal)$ a t-structure in $\Der(\ModR)$. 
	The t-structure is compactly generated if and only if $\Ucal = \Perp{0} \Ecal$ for some set $\Ecal$ of shifts of stalk complexes of injective $R$-modules.
\end{cor}
\begin{proof}
	If $(\Ucal,\Vcal)$ is compactly generated, then there is by Theorem~\ref{T:AJS} an sp-filtration $\phi$ on $\Spec(R)$ such that 
$$\Ucal = \{X \in \Der(\ModR) \mid \Supp H^n(X) \subseteq \phi(n) ~\forall n \in \Z\}.$$
 Then $\Ucal = \Perp{0}\Ecal$ for $\Ecal = \{E(R/\pp)[-n] \mid n \in \Z, \pp \not\in \phi(n)\}$. In fact, one can easily deduce this fact by noting that 
\begin{align}\label{star}\tag{$*$}
\Hom_{\Der(\ModR)}(X, E(R/\pp)[-n])=0\Longleftrightarrow \pp\notin \Supp H^n(X),
\end{align} see \cite[Theorems 18.4 and 18.6]{Matsu} and \cite[Theorem 13.4.1]{KaSc}. See also \cite[Lemma 3.2]{Hcom}.
	
Conversely, let $\Ecal_n$ be a set of injective $R$-modules for each $n \in \Z$ and write $\Ecal = \bigcup_{n \in \Z}\Ecal_n[-n]$. We assume that $\Ucal = \Perp{0}\Ecal$.
It is not hard to see that 
$$\Ucal = \{X \in \Der(\ModR) \mid H^n(X) \in \Ucal_n ~\forall n \in \Z\},$$ where 
$\Ucal_n = \{M \in \ModR \mid \Hom_R(M,E) = 0 ~\forall E \in \Ecal_n, ~\forall n \in \Z\}$. 
But then $\Ucal_n$ is a hereditary torsion class in $\ModR$, and hence there (uniquely) exists a specialization closed subset $U_ n$ such that $\Ucal_n = \{M \in \ModR \mid \Supp M \subseteq U_n\}$, see e.g. \cite[Proposition 2.3]{APST}. 
Since $\Ucal$ is closed under suspensions, we see that $\Ucal_{n+1}\subseteq \Ucal_{n}$, and this implies $U_{n+1} \subseteq U_{n}$. 
Then $\Ucal$ is given by an sp-filtration $\phi$ defined by $\phi(n) = U_n$ for all $n \in \Z$. In fact, it follows that 
$$\Ucal = \{X \in \Der(\ModR) \mid \Supp H^n(X) \subseteq U_n ~\forall n \in \Z\}.$$
Hence the t-structure $(\Ucal, \Vcal)$ is compactly generated by Theorem~\ref{T:AJS}.
\end{proof}

 In terms of the above fact, essential tasks are to give such a set $\Ecal$ and to show $\Ucal = \Perp{0} \Ecal$, for a given homotopically smashing $t$-structure $(\Ucal, \Vcal)$.


We next recall the notion of Milnor (co)limits. 
Consider a sequence of morphisms in $\Der(\ModR)$;
$$X_0 \xrightarrow{f_0} X_1 \xrightarrow{f_1} X_2 \xrightarrow{f_2} X_3 \xrightarrow{f_3} \cdots.$$
The \EMP{Milnor colimit} $\Mcolim_{n \geq 0} X_n$ of this sequence is the cone of the morphism $\bigoplus_{n \geq 0} X_n \xrightarrow{1-\sigma} \bigoplus_{n \geq 0} X_n$, where $\sigma$ is a morphism with components $X_k \xrightarrow{f_k} X_{k+1}$. 
Similarly, for a sequence of morphisms 
$$
X_0 \xleftarrow{f_0} X_1 \xleftarrow{f_1} X_2 \xleftarrow{f_2} X_3 \xleftarrow{f_3} \cdots,
$$
the \EMP{Milnor limit} $\Mlim_{n \geq 0}X_n$ is the cocone of the morphism $\prod_{n \geq 0} X_n \xrightarrow{1-\lambda} \prod_{n \geq 0} X_n$, 
where $\lambda$ has components $X_k \xrightarrow{f_k} X_{k-1}$ for $k > 0$ and a zero morphism $X_0 \rightarrow 0$.

What we need to observe for our purpose is that any aisle is closed under Milnor colimits and any coaisle is closed under Milnor limits by their definitions. See also \ref{hst}.
The relationship between Milnor (co)limits and homotopy (co)limits is explained in the following remark.

\begin{remark}\label{relationship}
In the literature, Milnor (co)limits have been also called homotopy (co)limits, see \cite{N3} for example. Here, we follow the custom of \cite{KN} where Milnor (co)limits are distinguished from homotopy (co)limits defined by derivators; unlike the latter, the former do not have functorial properties. 
However, as explained in \cite[Appendix 2]{KN}, Milnor colimits can be realized as homotopy colimits of directed systems indexed by $\mathbb{N}$, and vice versa. Since the standard derivator $\DD_{\ModR}$ is strong and stable (see Example \ref{example derivator}), the dual statement holds true as well by \cite[Proposition 6.27, Lemma 9.3 and Example 9.41(iii)]{Gro};  Milnor limits can be realized as homotopy limits of inverse systems indexed by $\mathbb{N}$, and vice versa.
\end{remark}

If $X$ is a cochain complex, we let $\sigma^{\geq n}(X)$ and $\sigma^{\leq n}(X)$ denote the right and left brutal truncations of $X$ in degree $n \in \Z$. We recall that there are canonical inductive and inverse systems in $\Ch(\ModR)$:
$$\sigma^{\geq 0}(X) \rightarrow \sigma^{\geq -1}(X) \rightarrow \sigma^{\geq -2}(X) \rightarrow \sigma^{\geq -3}(X) \rightarrow \cdots$$
and
$$\sigma^{\leq 0}(X) \leftarrow \sigma^{\leq 1}(X) \leftarrow \sigma^{\leq 2}(X) \leftarrow \sigma^{\leq 3}(X) \leftarrow \cdots,$$
where $X$ coincides with both their inductive limit and inverse limit.
Further, if we regard these systems as sequences in $\Der(\ModR)$, then $X$ is isomorphic to both $\Mcolim_{n \geq 0}\sigma^{\geq -n}(X)$ and to $\Mlim_{n \geq 0}\sigma^{\leq n}(X)$ in $\Der(\ModR)$, see \cite[Remarks 2.2 and 2.3]{N3} or \cite[Lemma 5.3 and Theorem A.2]{KN}. 

For $\pp\in \Spec (R)$, we denote by $\kappa(\pp)$ the residue field $R_\pp/ \pp R_\pp$ of $R_\pp$.
 
\begin{lem}\label{L00}
	Let $R$ be a commutative ring and $(\Ucal,\Vcal)$ a t-structure in $\Der(\ModR)$. 
If $\pp \in \Spec(R)$ and $n \in \Z$ satisfy $\Hom_{\Der(\ModR)}(\kappa(\pp)[-n],Y ) \neq 0$ for some $Y \in \Vcal$, then $\kappa(\pp)[-n] \in \Vcal$.
\end{lem}
\begin{proof}
	Put $X = \RHom_R(\kappa(\pp),Y)$, then $X$ is isomorphic to a complex of vector spaces over the residue field $\kappa(\pp)$ in $\Der(\ModR)$. In particular, $X$ is isomorphic to $\bigoplus_{i \in \Z}H^i(X)[-i]$ in $\Der(\ModR)$. Using the assumption, 
$$H^n(X) = H^n \RHom{}_R(\kappa(\pp),Y) \simeq \Hom{}_{\Der(\ModR)}(\kappa(\pp)[-n],Y) \neq 0.$$ 
Now, it follows from \cite[Proposition 2.2(ii)]{Hcom} that the complex $X$ belongs to $\Vcal$ because $Y\in \Vcal$.
Hence, its direct summand $H^n(X)[-n]$ also belongs to $\Vcal$. As $H^n(X)$ is a non-zero vector space over $\kappa(\pp)$, we conclude that $\kappa(\pp)[-n] \in \Vcal$.
\end{proof}

\begin{cor}\label{L01}
Let $R$ be a commutative ring and $(\Ucal,\Vcal)$ a t-structure in $\Der(\ModR)$.
For any $n \in \Z$ and any $\pp \in \Spec(R)$, either $\kappa(\pp)[-n] \in \Ucal$ or $\kappa(\pp)[-n] \in \Vcal$.
\end{cor}
\begin{proof}
	Suppose that $\kappa(\pp)[-n] \not\in \Ucal$, and consider the approximation triangle 
$$X \rightarrow \kappa(\pp)[-n] \xrightarrow{f} Y \rightarrow X[1]$$
with respect to the t-structure $(\Ucal,\Vcal)$. Since $\kappa(\pp)[-n] \not\in \Ucal$ the map $f$ is non-zero in $\Der(\ModR)$, as otherwise $\kappa(\pp)[-n]$ would be a direct summand of $X \in \Ucal$. Therefore, $\Hom_{\Der(\ModR)}(\kappa(\pp)[-n],Y) \neq 0$, and thus $\kappa(\pp)[-n] \in \Vcal$ by Lemma~\ref{L00}.
\end{proof}

\emph{From now on, let $(\Ucal,\Vcal)$ be a homotopically smashing t-structure in the derived category $\Der(\ModR)$ of a commutative noetherian ring $R$. 
Hence the coaisle $\Vcal$ is closed under both homotopy limits and directed homotopy colimits.}

\begin{lem}\label{L02}
	For any $n \in \Z$ and any $\pp \in \Spec(R)$, we have  
	$$\kappa(\pp)[-n] \in \Vcal \iff E(R/\pp)[-n] \in \Vcal.$$
\end{lem}
\begin{proof}
We first remark that $E(R/\pp)$ admits a filtration by coproducts of copies of $\kappa(\pp)$, see \cite[Theorem 3.4]{Ma}. Since $\Vcal$ is closed under extensions and directed homotopy colimits, we see that $\kappa(\pp)[-n] \in \Vcal$ implies $E(R/\pp)[-n] \in \Vcal$. 

For the other implication, note that
$$\Hom_{\Der(\ModR)}(\kappa(\pp)[-n],E(R/\pp)[-n])\simeq \Hom_{R}(\kappa(\pp),E(R/\pp))\neq 0$$
see \cite[Theorem 18.4]{Matsu}.
Thus $E(R/\pp)[-n]\in \Vcal$
implies $\kappa(\pp)[-n]\in \Vcal$ by Lemma \ref{L00}.
\end{proof}

\begin{lem}\label{L03}
	Let $\qq \subseteq \pp$ be prime ideals of $R$ and let $n \in \Z$. Then $\kappa(\pp)[-n] \in \Vcal$ implies $\kappa(\qq)[-n] \in \Vcal$.
\end{lem}
\begin{proof}
	Consider a commutative diagram
$$
	\begin{CD}
		R/\qq @>>> R/\pp \\
		@V\subseteq VV @VV\subseteq V \\
		\kappa(\qq) @>>> E(R/\pp) 
	\end{CD}
$$
where the top map is the canonical surjection, the vertical maps are the essential inclusions, and the bottom map is obtained from the injectivity of $E(R/\pp)$. In particular, the bottom map is non-zero, and thus $\Hom_{\Der(\ModR)}(\kappa(\qq)[-n],E(R/\pp)[-n]) \neq 0$. As $\kappa(\pp)[-n] \in \Vcal$, we have $E(R/\pp)[-n] \in \Vcal$ by Lemma~\ref{L02}, and therefore $\kappa(\qq)[-n] \in \Vcal$ by Lemma~\ref{L00}.
\end{proof}

We define a set of shifts of stalk complexes of injective modules as follows: 
$$\Ecal = \{E(R/\pp)[-n] \mid n\in \mathbb{Z}, \pp\in \Spec (R), E(R/\pp)[-n]\in \Vcal\}.$$
The goal is to show that $\Ucal=\Perp{0}\Ecal$, as then $(\Ucal,\Vcal)$ is compactly generated by Corollary~\ref{C01}.

Let us define $\Wcal$ to be the smallest subcategory of $\Der(\ModR)$ satisfying the following three properties:
	\begin{enumerate}
		\item[(i)] $\Ecal \subseteq \Wcal$;
		\item[(ii)] $\Wcal$ is closed under extensions, cosuspensions and arbitrary products;
		\item[(iii)] $\Wcal$ is closed under directed homotopy colimits.
	\end{enumerate}
Compare the conditions (i) and (ii) with \cite[1.2]{AJS} or the last paragraph of \cite[\S 3.1]{A}.
The condition (iii) makes sense as we are interested in homotopically smashing t-structures.
Note that (ii) implies that $\Wcal$ is closed under Milnor limits.
In terms of Remark \ref{relationship} (or Fact \ref{fact}), (iii) implies that $\Wcal$ is closed under Milnor colimits.

\begin{lem}\label{L04}
	If $\Vcal \subseteq \Wcal$, then the t-structure $(\Ucal,\Vcal)$ is compactly generated.
\end{lem}

\begin{proof}
	Let $\Ucal' = \Perp{0}\Ecal$ and
	$U_n  = \{\pp \in \Spec(R) \mid E(R/\pp)[-n] \not\in \Vcal\}$ for each $n \in \Z$.
It is seen from Lemmas \ref{L02} and \ref{L03} that each  
$U_n$ is specialization closed.
By using \eqref{star}, we can deduce that 
$$\Ucal'=\{X \in \Der(\ModR) \mid \Supp H^n(X) \subseteq U_n ~\forall n \in \Z\}.$$
As the coaisle $\Vcal$ is closed under cosuspensions,
it follows that $U_{n+1}\subseteq U_n$ for each $n\in \Z$. Hence, the family of the specialization closed subsets $U_n$ naturally gives an sp-filtration.
Then, Theorem~\ref{T:AJS} implies that $\Ucal'$ is the aisle of a compactly generated t-structure $(\Ucal',\Vcal')$ in $\Der(\ModR)$. Since $\Ecal \subseteq \Vcal$, we clearly have $\Vcal' \subseteq \Vcal$. Furthermore, as $\Vcal'$ is the coaisle of a compactly generated t-structure, $\Vcal'$ is closed under extensions, products, and directed homotopy colimits, and therefore $\Ecal \subseteq \Vcal'$ yields  a chain of inclusions $\Wcal \subseteq \Vcal' \subseteq \Vcal$. Therefore, the assumption $\Vcal \subseteq \Wcal$ implies that $\Vcal = \Vcal'$, and whence the t-structure $(\Ucal,\Vcal)$ coincides with the compactly generated t-structure $(\Ucal',\Vcal')$.
\end{proof}
The following is an easy consequence of the closure property of both $\Ucal$ and $\Vcal$ under directed homotopy colimits, see Fact \ref{fact}, \ref{hst} and the proof of \cite[Lemma 4.1]{Hcom}.
\begin{lem}\label{L05}
	For any $\pp \in \Spec(R)$, both subcategories $\Ucal$ and $\Vcal$ are closed under applications of the localization functor $- \otimes^{\mathbf{L}}_R R_{\pp} =- \otimes_R R_{\pp} $.
\end{lem}

We recall here some basic notions and facts. Let $V$ be a specialization closed subset of $\Spec (R)$. For an $R$-module $M$, we set $\Gamma_V(M)=\{ x\in M \mid \Supp xR \subseteq V\}$. Then $\Gamma_V$ is a left exact functor on the category $\ModR$ of $R$-modules, and it yields a right derived functor $\mathbf{R}\Gamma_V: \Der(\ModR) \rightarrow \Der(\ModR)$.
When $V$ is the Zariski closed subset $V(I)$ for an ideal $I$ of $R$, $\Gamma_{V(I)}$ coincides with the $I$-torsion functor $\Gamma_{I}=\varinjlim_{n\geq 1}\Hom_{R}(R/I^n,-)$ see \cite[\S 3.5]{Lip} and \cite[II; Exercise 5.6]{Har} for more details. 

It is well-known that any injective $R$-module is of the form $\bigoplus_{\pp \in \Spec(R)}E(R/\pp)^{(\varkappa_{\pp})}$, where $\varkappa_{\pp}$ is some cardinal and $E(R/\pp)^{(\varkappa_{\pp})}$ is the coproduct of $\varkappa_{\pp}$-copies of 
$E(R/\pp)$, see \cite[Theorem 18.5]{Matsu}.
As mentioned in the proof of {\em loc. cit.}, the localization of $\bigoplus_{\pp \in \Spec(R)}E(R/\pp)^{(\varkappa_{\pp})}$
at $\pp$ is $\bigoplus_{\qq\subseteq \pp}E(R/\qq)^{(\varkappa_{\qq})}$.
Furthermore, if $V$ is a specialization closed subset $V$ of $\Spec (R)$, then we see that $\Gamma_V$ sends $\bigoplus_{\pp \in \Spec(R)}E(R/\pp)^{(\varkappa_{\pp})}$ to 
$\bigoplus_{\pp\in V}E(R/\pp)^{(\varkappa_{\pp})}$.

Now, we fix an object $X \in \Vcal$, and our aim is to show that $X \in \Wcal$. Then the t-structure $(\Ucal,\Vcal)$ will be compactly generated
by Lemma \ref{L04}. Note that we may assume that $X$ is a complex of injective $R$-modules.
Moreover, for each $\pp \in \Spec(R)$, define a complex $X(\pp)$ as $\Gamma_{\pp}(X \otimes_R R_{\pp})$. It is a subquotient complex of $X$ whose components are coproducts of copies of $E(R/\pp)$.

\begin{lem}\label{L06}
	For each $\pp \in \Spec(R)$, $X(\pp) \in \Wcal$.
\end{lem}
\begin{proof}
	Put $Z = \RHom_R(\kappa(\pp),X(\pp))$. Since $X$ is a complex of injective modules, $X(\pp)=\Gamma_{\pp}(X_\pp) \simeq \RGamma_{\pp}(X_\pp)$ by \cite[Lemma 3.5.1]{Lip}. Therefore, using the adjunction of $\RGamma_{\pp}$, we get that
\begin{align}\label{star2}\tag{$**$}
Z = \RHom{}_R(\kappa(\pp),\mathbf{R}\Gamma_{\pp}(X_\pp)) \simeq  \RHom{}_R(\kappa(\pp),X_\pp),
\end{align}
see \cite[Proposition 3.5.4(ii)]{Lip}. 
Furthermore, by a result of Foxby and Iyengar \cite[Theorem 2.1]{FI}, we have the following equality:
	$$\inf \mathbf{R}\Gamma_{\pp}(X(\pp)) = \inf \RHom_{R}(\kappa(\pp),X(\pp)),$$
where we implicitly used two standard isomorphisms
$$\mathbf{R}\Gamma_{\pp}(X(\pp))\cong \mathbf{R}\Gamma_{\pp R_\pp}(X(\pp)) \ \text{and}\  \RHom_{R}(\kappa(\pp),X(\pp))\cong \RHom_{R_\pp}(\kappa(\pp),X(\pp)).$$
As $X(\pp) \simeq \mathbf{R}\Gamma_{\pp}(X(\pp))$ and $Z=\RHom{}_{R}(\kappa(\pp),X(\pp))$, it follows that 
$$\inf X(\pp) = \inf Z.$$
If $\inf Z= \infty$, that is, $Z=0$, then $X(\pp)$ must be zero in $\Der(\ModR)$, because $X(\pp)$ belongs to the smallest localizing subcategory  of $\Der(\ModR)$ containing $\kappa(\pp)$, see \cite[Lemma 2.9]{N1}.
Hence $X(\pp)\in \Wcal$ in this case.

Suppose that $\inf Z< \infty$. 
Note that $X_\pp \in \Vcal$ by Lemma~\ref{L05}.
Using \eqref{star2}, Lemma~\ref{L00} and Lemma \ref{L02}, we conclude that for any $n \in \Z$
\begin{equation*}
\inf Z \leq n \Longrightarrow \kappa(\pp)[-n] \in \Vcal \iff E(R/\pp)[-n]\in \Vcal \stackrel{\text{def}}{\iff} E(R/\pp)[-n] \in \Ecal.
	\end{equation*}
Therefore, if $\inf Z = - \infty$, then $E(R/\pp)[n] \in \Ecal$ for all $n \in \Z$. As $\Wcal$ is closed under extensions, coproducts, and both Milnor limits and colimits, a standard argument using both left and right brutal truncations shows $X(\pp) \in \Wcal$.

When $\inf Z \in \Z$, we may replace $X(\pp)$ by a complex concentrated in degrees greater than or equal to $\inf Z$, where its components are coproducts of copies of $E(R/\pp)$, see e.g. \cite[Proposition 2.1]{CI}. Since $E(R/\pp)[-n] \in \Ecal$ for all $n \geq \inf Z$, the closure property of $\Wcal$ under extensions, coproducts, and Milnor limits imply $X(\pp) \in \Wcal$.
\end{proof}
\begin{lem}\label{L07}
	The complex $X$ belongs to $\Wcal$.
\end{lem}
\begin{proof}
	First, let $V_0$ be the set of all maximal prime ideals in $\Spec(R)$. Then it is elementary to see that $\Gamma_{V_0}(X) = \bigoplus_{\pp \in V_0}X(\pp)$, and therefore $\Gamma_{V_0}(X)$ belongs to $\Wcal$ by Lemma~\ref{L06}.

	Let $\Vscr$ be the poset of all specialization closed subsets $V$ of $\Spec(R)$ such that $\Gamma_{V}(X) \in \Wcal$, ordered by inclusion. Then $\Vscr$ is inductive. Indeed, $V_0 \in \Vscr$, and for every increasing chain $(V_\alpha \mid \alpha<\lambda)$ of elements of $\Vscr$, the direct limit of the induced directed system $(\Gamma_{V_\alpha}(X) \mid \alpha<\lambda)$ belongs to $\Wcal$, because $\Wcal$ is closed under directed homotopy colimits. 

	Then Zorn's Lemma applies to $\Vscr$, so let $V$ be a maximal element of $\Vscr$. Towards a contradiction, assume that there is a prime $\pp \in \Spec(R) \setminus V$. Since $R$ is noetherian, we can assume that $\pp$ is a maximal element in $\Spec(R) \setminus V$, and then $V' = V \cup \{\pp\}$ is a specialization closed subset of $\Spec(R)$. 	
Note that there is a canonical degreewise split exact sequence
$$0 \rightarrow \Gamma_{V}(X) \rightarrow \Gamma_{V'}(X) \rightarrow X(\pp) \rightarrow 0$$
in the category of cochain complexes. As $\Wcal$ is closed under extensions in $\Der(\ModR)$ and $X(\pp)\in \Wcal$ by Lemma \ref{L06}, it follows that $\Gamma_{V'}(X) \in \Wcal$, establishing the contradiction with the maximality of $V$.
In conclusion, $X = \Gamma_{\Spec(R)}(X) \in \Wcal$, as desired.
\end{proof}
Combining Lemma~\ref{L07} with Lemma~\ref{L04}, we have proved Theorem~\ref{T01}. We conclude this section with a couple of particular consequences of Theorem~\ref{T01} to the cosilting theory of $\Der(\ModR)$. The terminology and basic facts about cosilting objects are explained in \ref{SS:cosilting}. The following result should be compared with \cite[Theorem 7.8]{A}.

\begin{cor}\label{C:cosilting}
	Let $R$ be a commutative noetherian ring. Then any pure-injective cosilting object in $\Der(\ModR)$ is of cofinite type.
\end{cor}
\begin{proof}
	Let $C \in \Der(\ModR)$ be a pure-injective cosilting object. Using the facts from \ref{SS:cosilting}, the t-structure $(\Perp{\leq 0}C, \Perp{>0}C)$ induced by $C$ is homotopically smashing. By Theorem~\ref{T01}, this t-structure is compactly generated, which establishes that $C$ is of cofinite type by definition.
\end{proof}

\begin{cor}\label{C:bounded}
	Let $R$ be a commutative noetherian ring. Then any pure-injective cosilting object in $\Der(\ModR)$ is cohomologically bounded below. That is, for any pure-injective cosilting object $C \in \Der(\ModR)$, there is an integer $m$ such that $H^i(C)=0$ for $i\leq m$.
\end{cor}
\begin{proof}
	Let $(\Ucal,\Vcal) = (\Perp{\leq 0}C, \Perp{>0}C)$ be the t-structure induced by a pure-injective cosilting object $C \in \Der(\ModR)$. By Corollary~\ref{C:cosilting} and Theorem~\ref{T:AJS}, there is an sp-filtration $\phi$ of $\Spec(R)$ such that $\Ucal = \{X \in \Der(\ModR) \mid \Supp H^n(X) \subseteq \phi(n) ~\forall n \in \Z\}$. As the t-structure $(\Ucal,\Vcal)$ is non-degenerate (see \ref{SS:cosilting}), the sp-filtration $\phi$ satisfies $\bigcup_{n \in \Z} \phi(n) = \Spec(R)$ and $\bigcap_{n \in \Z} \phi(n) = \emptyset$, see \cite[Remark 7.7(2)]{A}.

	Since $R$ is noetherian, the set of minimal prime ideals in $\Spec(R)$ is finite. Hence the condition $\bigcup_{n \in \Z}\phi(n) = \Spec(R)$ implies that there is $m \in \Z$ such that $\phi(m) = \Spec(R)$. It then follows that $C \in \Der^{>m}$ (see Example \ref{standard t-structure}) because 
$C \in \Vcal$ by definition and $\Vcal = \{X \in \Der(\ModR) \mid \mathbf{R}\Gamma_{\phi(n)}(X) \in \Der{}^{>n}  ~\forall n \in \Z\}$ by \cite[Theorem 3.11]{AJS}.
\end{proof}


\appendix
\section{Telescope conjecture for homotopically smashing t-structures in Grothendieck derivators}

\begin{A} {\bf Homotopy (co)limits in triangulated categories.}  
Let $\Cat$ stand for the 2-category of all small categories and $\CAT$ be the ``2-category'' of all categories. A \EMP{derivator} is a 2-functor $\DD: \Cat^{\op} \rightarrow \CAT$ satisfying certain conditions, we refer the reader to \cite{Gro} and references therein for the precise definition and for the introduction to the theory of derivators. Let $\star \in \Cat$ denote the category with a single object and a single map. Then we call $\DD(\star)$ the \EMP{underlying} category of the derivator $\DD$. For every small category $I \in \Cat$ we denote the unique functor $I \rightarrow \star$ by $\pi_I$. It is a part of the definition of a derivator that the induced functor $\DD(\pi_I): \DD(\star) \rightarrow \DD(I)$ admits both the  left and right adjoint functors:
		$$
	\begin{tikzcd}[row sep=large, column sep=huge]
		\DD(\star)	 \arrow{r}{\DD(\pi_I)} & \DD(I)  \arrow[bend right]{l}[above]{\hocolim_I}  \arrow[bend left]{l}[below]{\holim_I}
	\end{tikzcd}
		$$
The right adjoint to $\DD(\pi_I)$ is denoted by $\holim_I: \DD(I) \rightarrow \DD(\star)$ and is called the \EMP{homotopy limit} functor, while the left adjoint is denoted by $\hocolim_I: \DD(I) \rightarrow \DD(\star)$ and called the \EMP{homotopy colimit} functor. 

 For any $I \in \Cat$ and any object $i \in I$, we let $\iota_i: \star \rightarrow I$ be the unique functor sending the only object of $\star$ to $i \in I$. The functors $\DD(\iota_i): \DD(I) \rightarrow \DD(\star)$ induce a functor $\dia_I: \DD(I) \rightarrow \DD(\star)^I$ called the \EMP{diagram functor}. It is a part of the motivation for the theory of derivators that the diagram functor is rarely an equivalence. The usual terminology refers to objects of $\DD(I)$ as \EMP{coherent diagrams} of shape $I$. For any coherent diagram $\Xscr \in \DD(I)$, we call $\dia_I(\Xscr)$ the \EMP{(incoherent) diagram} underlying the coherent diagram $\Xscr$.

 In our context, we require that our derivator is in addition \EMP{strong} and \EMP{stable}, for precise definitions see \cite{Gro}. Here, we just remark that the condition of being strong allows us to lift incoherent diagrams to coherent diagrams along the diagram functor for certain special shapes, while the stability condition implies that $\DD(I)$ is a triangulated category for all $I \in \Cat$, and that the homotopy limit and colimit functors are triangulated. 

Let $\DD$ be a strong and stable derivator. We say that a subcategory $\Ccal$ of $\DD(\star)$ is \EMP{closed under homotopy colimits} if $\hocolim_I \Xscr \in \Ccal$ for any small category $I$ and any coherent diagram $\Xscr \in \DD(I)$ such that all components of the underlying diagram $\dia_I(\Xscr)$ belong to $\Ccal$. A subcategory \EMP{closed under homotopy limits} is defined analogously. We say that a subcategory of $\DD(\star)$ is \EMP{closed under directed homotopy colimits} if it is closed under homotopy colimits defined over directed small categories $I$.

 Finally, a derivator $\DD$ is \EMP{compactly generated} if it is strong and stable, and if the triangulated category $\DD(\star)$ is compactly generated. As a consequence, the triangulated category $\DD(I)$ is compactly generated for any $I \in \Cat$, see \cite[Lemma 3.2]{L}. 
\end{A}

\begin{exampleA}\label{example derivator}
	Let $\Gcal$ be a Grothendieck category. For any small category $I$ consider $\Gcal^I$, the Grothendieck category of all $I$-shaped diagrams in $\Gcal$. The derived category $\Der(\Gcal^I)$ of $\Gcal^I$ can be naturally considered as the Verdier localization of the diagram category $\Ch(\Gcal)^I$ of cochain complexes of objects of $\Gcal$. Then the assignment $\DD_{\Gcal}: \Cat \rightarrow \CAT$ given by $I \mapsto \Der(\Gcal^I)$ naturally extends to a strong and stable derivator called the \EMP{standard derivator} on the Grothendieck category $\Gcal$, see \cite[\S 5]{St} for details. The constant diagonal functor $\Delta_I: \Ch(\Gcal) \rightarrow \Ch(\Gcal)^I$ is exact, and its left and right adjoints are the usual colimit and limit functors on the Grothendieck category $\Ch(\Gcal)$ respectively. Deriving this adjunction yields the following adjunction picture:
		$$
	\begin{tikzcd}[row sep=large, column sep=huge]
		\Der(\Gcal)	 \arrow{r}{\Delta_I} & \Der(\Gcal^I)  \arrow[bend right]{l}[above]{\Lcolim_I}  \arrow[bend left]{l}[below]{\Rlim_I}
	\end{tikzcd}
		$$
	Comparing this adjunction with the construction of the derivator $\DD_{\Gcal}$ in \cite[\S 5]{St}, it follows that in this situation, the homotopy limit and colimit functors yielded by the derivator $\DD_{\Gcal}$ are naturally equivalent to the derived limit and colimit functors on $\Der(\Gcal^I)$: $\holim_I = \Rlim_I$ and $\hocolim_I = \Lcolim_I$. The diagram functor 
	$$\dia_I: \Der(\Gcal^I) \rightarrow \Der(\Gcal)^I$$ 
just sends a coherent diagram of $\Der(\Gcal^I)$ to an ordinary $I$-shaped diagram of $\Der(\Gcal)^I$ in a natural way. 

	Let $I$ be a directed small category. Then the direct limit functor $\varinjlim_I$ is exact in the Grothendieck category $\Ch(\Gcal)$, and therefore we get for any coherent diagram $\Xscr \in \Der(\Gcal^I)$ an isomorphism $\hocolim_I(\Xscr) = \Lcolim_I (\Xscr) \simeq \varinjlim_I \Xscr$ in $\Der(\Gcal)$, where the last direct limit is computed in $\Ch(\Gcal)$ (cf. \cite[The proof of Proposition 6.6]{St}. Hence a subcategory of $\Der(\Gcal)$ is closed under directed homotopy colimits if and only if it is closed under direct limits computed in the category $\Ch(\Gcal)$ of cochain complexes.
\end{exampleA}

A standard derivator on the category of modules over a ring is compactly generated. A more general source of examples of compactly generated derivators comes from (compactly generated) stable model categories, see \cite[Example 4.2(1)]{Gro2}. Such examples include derived categories of small dg categories and the stable homotopy category of spectra.

\begin{A}{\bf Homotopically smashing t-structures.}\label{hst}
	Let $\DD$ be a strong and stable derivator and let $(\Ucal,\Vcal)$ be a t-structure in $\DD(\star)$. Then the aisle $\Ucal$ is closed under homotopy colimits and the coaisle $\Vcal$ is closed under homotopy limits; this is \cite[Proposition 4.2]{SSV}. Following \cite{SSV}, we say that a t-structure $(\Ucal,\Vcal)$ is \EMP{homotopically smashing} if the coaisle $\Vcal$ is closed under directed homotopy colimits. 
Any compactly generated t-structure $(\Ucal,\Vcal)$ is homotopically smashing, see \cite[Proposition 5.6]{SSV}.
\end{A}


\begin{exampleA}\label{standard t-structure}
 Let us give a typical example of homotopically smashing t-structures in the case of the standard derivator $\DD_{\Gcal}$ of a Grothendieck category $\Gcal$.
For any $n \in \Z$, there is a \EMP{standard} t-structure $(\Der{}^{\leq n},\Der{}^{>n})$ in the derived category $\Der(\Gcal) = \DD_{\Gcal}(\star)$, where $\Der{}^{\leq n}$ (resp. $\Der{}^{> n}$) stands for the subcategory of complexes $X$ with $H^i(X) = 0$ for $i>n$ (resp. $i\leq n$).
The approximation triangle induced by this t-structure is given by the soft truncations $\tau^{\leq n}$ and $\tau^{>n}$ of cochain complexes in degree $n$. Let $I$ be a directed small category and $\Xscr \in \Der(\Gcal^I)$ a coherent diagram. Let us fix an ordinary diagram $\{X_i\}_{i \in I} \in \Ch(\Gcal^I) \simeq \Ch(\Gcal)^I$ representing $\Xscr$. Since $H^n(\hocolim_I \Xscr) \simeq \varinjlim_I H^n(X_i)$, it is easy to see that the standard t-structure $(\Der{}^{\leq n},\Der{}^{>n})$ is homotopically smashing. Furthermore, if $\Gcal$ admits a small projective generator (and therefore, is equivalent to a module category), the derivator $\DD_{\Gcal}$, as well as the t-structure $(\Der{}^{\leq n},\Der{}^{>n})$, are easily seen to be compactly generated.
\end{exampleA}


\begin{A}{\bf  Smashing localizations.} \label{SS:stable}
Note that if $(\Ucal,\Vcal)$ is a homotopically smashing t-structure, then in particular $\Vcal$ is closed under coproducts. The converse implication is often not true as demonstrated in \cite[Example 6.2]{SSV}. This distinction is not visible among the t-structures coming from localizing theory, as we shortly discuss below.

A t-structure $(\Ucal,\Vcal)$ in a triangulated category $\Tcal$ is called \EMP{stable} if $\Ucal[-1] \subseteq \Ucal$, or equivalently, if both $\Ucal$ and $\Vcal$ are triangulated subcategories of $\Tcal$. Assume that $\Tcal$ is an underlying category of a compactly generated derivator. The aisles of such t-structures are precisely the kernels of Bousfield localizations of $\Tcal$, that is, triangulated coreflective subcategories of $\Tcal$. Recall that a subcategory $\Lcal$ of $\Tcal$ is a \EMP{smashing subcategory} if it is a kernel of a Bousfield localization functor which preserves all coproducts. Then a stable t-structure $(\Ucal,\Vcal)$ in $\Tcal$ is homotopically smashing if and only if $\Ucal$ is a smashing subcategory of $\Tcal$. This follows from \cite[Theorem A]{Kr}, see also \cite[Proposition 5.6]{SSV}, \cite[Theorem 3.12]{L} and \cite[Proposition 6.3]{LV}.
\end{A}


\begin{A}{\bf Telescope conjecture for homotopically smashing t-structures.}
Given a compactly generated triangulated category $\Tcal$, the telescope conjecture is a question which asks whether all smashing subcategories are compactly generated. In the light of \ref{SS:stable}, when $\Tcal$ is the underlying category of a compactly generated derivator, the telescope conjecture asks equivalently whether any stable homotopically smashing t-structure is compactly generated. The following is a natural generalization of this question to t-structures which are not necessarily stable.
\end{A}


\begin{questA}\label{Q0}
For which compactly generated derivator $\DD$ is every homotopically smashing t-structure in $\DD(\star)$ compactly generated?
\end{questA}

Theorem \ref{T01} shows that  Question~\ref{Q0} has an affirmative answer in the case of the standard derivator $\DD_{\ModR}$ of Example~\ref{example derivator} for the module category of a commutative noetherian ring $R$. 

\begin{A}{\bf Cosilting t-structures.}\label{SS:cosilting}
We discuss another important source of examples of homotopically smashing t-structures which comes from a general version of tilting theory. Let $\DD$ be a compactly generated derivator. If $(\Ucal,\Vcal)$ is a t-structure in $\DD(\star)$, the category $\Hcal = \Ucal[-1] \cap \Vcal$ is called the \EMP{heart} of the t-structure, and is always an abelian category and there is a cohomological functor from $\DD(\star)$ to $\Hcal$, \cite{BBD}. 

	A t-structure $(\Ucal,\Vcal)$ is \EMP{non-degenerate} if $\bigcap_{n \in \Z}\Ucal[n] = 0 = \bigcap_{n \in \Z}\Vcal[n]$. This condition characterizes the case in which the cohomological functor to the heart of $(\Ucal,\Vcal)$ is conservative. Following \cite{PV} and \cite{NSZ}, we call an object $C$ of $\DD(\star)$ \EMP{cosilting} if the pair $(\Perp{\leq 0}C,\Perp{>0}C)$ forms a t-structure in $\DD(\star)$. Such a t-structure is non-degenerate and its heart has all coproducts and admits an injective cogenerator, \cite[Theorem 3.5]{AMV2}.

	Let $(\Ucal,\Vcal)$ be a non-degenerate t-structure such that $\Vcal$ is closed under all coproducts. Then the heart $\Hcal$ is a Grothendieck category if and only if $(\Ucal,\Vcal)$ is homotopically smashing and this is further equivalent to the t-structure being induced by a cosilting object which is pure-injective in $\DD(\star)$ in the sense of \cite[\S 1]{Kr} (see \cite[Corollary 3.8]{AMV2} and \cite[Theorem 4.6]{L}). In conclusion, non-degenerate homotopically smashing t-structures parametrize pure-injective cosilting objects in $\DD(\star)$ (up to a suitable equivalence). Furthermore, in many situations the assumption of pure-injectivity is redundant, see \cite[Remark 3.11(1)]{MV}. 

Following \cite[\S 7]{A} and \cite{AH}, we say that a cosilting object $C \in \DD(\star)$ is \EMP{of cofinite type} if the t-structure $(\Perp{\leq 0}C,\Perp{>0}C)$ induced by it is compactly generated. Any cosilting object of cofinite type is pure-injective, but the converse is not always true (see \cite{BH}). The terminology extends the one from the theory of $n$-cotilting modules, see e.g. \cite{APST}, \cite{AS}. An affirmative answer to Question~\ref{Q0} then yields that all pure-injective cosilting objects in $\DD(\star)$ are of cofinite type. 

It has been recently shown in \cite[Theorem 1.6]{SS} that the heart of any compactly generated t-structure is not just Grothendieck, but even locally finitely presented Grothendieck.
\end{A}

\begin{ac}
This project was started during the first author's visit to the University of Verona, supported by MSM100191801, and it was completed during the second author's visit to the Charles University. They are grateful to the universities for the hospitality and for providing excellent environments. They also thank Rosanna Laking for helpful comments and conversations on this paper.
\end{ac}


\bibliographystyle{abbrv}

\end{document}